\documentclass[10pt]{amsart}
\usepackage[utf8]{inputenc}					
\usepackage[T1]{fontenc}					
\usepackage{lmodern}						
\usepackage{microtype}
\usepackage{amssymb,amsthm}
\usepackage[final,breaklinks]{hyperref}
\usepackage[shortlabels]{enumitem}
\usepackage{mathtools}

\newcommand{\QH}{\mathbb{H}}
\newcommand{\R}{{\mathbb{R}}}

\newcommand{\interno}[2]{\left<#1,#2\right>}	
\newcommand{\set}[2]{\left\{#1 : #2\right\}}	
\DeclareMathOperator{\spn}{span}

\newtheoremstyle{fancy}
	{}
	{}
	{\slshape}
	{}
	{\scshape}
	{.}
	{.5em}
	{\thmname{#1}\thmnumber{ #2}\thmnote{\textbf{ (#3)}}}

\newcounter{theorems}
\newcounter{maintheorems}

\theoremstyle{fancy}

\newtheorem{prop}[theorems]{Proposition}
\newtheorem{thm}[theorems]{Theorem}
\newtheorem*{conj}{Conjecture}
\newtheorem*{example}{Example}
\newtheorem{main}[maintheorems]{Theorem}

\title{Conformal immersion of Riemannian products in low codimension}

\author{Felippe Guimar\~aes}
\author{Bruno Mendon\c{c}a}

\address[Felipe Guimar\~aes]{This study was financed in part by the Coordenação de Aperfeiçoamento de Pessoal de Nível Superior – Brasil (CAPES) – Finance Code 001}

\begin{document}

\maketitle

\begin{abstract}
	We proved that a conformal immersion of $M_0^{n_0}\times M_1^{n_1}$ as an hypersurface in a Euclidean space must be an extrinsic product of immersions, under the assumption that $n_0, n_1 \geq 2$ and that $M^{n_0}_0\times M^{n_1}_1$ is not conformally flat. We also stated a similar theorem for an arbitrary number of factors, more precisely, a conformal immersion $f\colon M^{n_0}_0 \times \cdots \times M^{n_k}_k \rightarrow \R^{n+k}$ must be an extrinsic product of immersions if one of the factors admits a plane with vanishing curvature and the remaining factors are not flat.
\end{abstract}

The simplest way to constructing an immersion of a Riemannian product manifold into a Euclidean space is to take an extrinsic product of immersions, that is, a product of immersions of each factor of the product manifold. A natural problem in the submanifold theory is to provide sufficient conditions for an immersion of a Riemannian product to be decomposed as an extrinsic product, and there are plenty of works in this subject, for instance,  \cite{eijiriMinimal,barbosaDajczerTojeiro,dajczerVlachosSnull,alexanderMaltzGlobal, Moore1971, dajczerTojeiroWarpedCod2, tojeiroConfNolk, nolkerRhamtype}. Some of them ask for intrinsic proprieties to decompose an arbitrary manifold, without the initially assumption that is already a product, in structures even more general than Riemannian products, as in \cite{tojeiroRhamconf, tojeiroRhamtype}. 

In low codimension, it turns out that there is not enough space  to the immersion of a Riemannian product manifold to be very complicated. Moore, in his outstanding work \cite{Moore1971}, showed that an immersion of a Riemannian product must be an extrinsic product when the factors are not flat and the codimension is the minimal possible. More precisely, he showed that 

\begin{thm}[Moore \cite{Moore1971}]\label{mooreTheo}
	Let $f\colon M^n \rightarrow \R^{n+k}$ be an isometric immersion of a Riemannian product manifold $M^n =\prod_{i=0}^{k} M^{n_i}_i $, where $n =\sum_{i=0}^{k} n_i$ and $n_i \geq 2$ for all $1 \leq i \leq k$. If the subset of points of $M_i^{n_i}$ at which all sectional curvatures vanish has empty interior for all $1 \leq i \leq k$, then $M^{n_0}_0$ is flat and $f(M)$ is an open subset of a $n_0$-cylinder over an extrinsic product of hypersurface immersions.
\end{thm}

The first step of his proof was to reduce the question to one of algebraic nature by proving an extrinsic decomposition De Rham type theorem. In the conformal realm, Tojeiro \cite{tojeiroConfNolk} proved an analogous theorem. In short, the mentioned works of Tojeiro and Moore was the main motivation of this paper. We adapted the technique used by Moore to obtain the following theorem:

\begin{main}\label{main2factors}
	Let $f \colon M^n \rightarrow \R^{n+1}$ be a conformal immersion of a Riemannian product manifold $M^n =M^{n_0}_0 \times M^{n_1}_1 $, where $n =n_0+n_1$ and $n_i \geq 2$ for $0 \leq i \leq 1$. If $M^n$ is not conformally flat at any point, then one of the following holds
	\begin{enumerate}[(i)]
		\item There exist an extrinsic product $\tilde{f}\colon M^n \rightarrow \R^{n+1}$ of isometric immersions, a homothety $H$ of $\R^{n+1}$ and an inversion $I$ in $\R^{n+1}$ with respect to a unit sphere such that
		$$f = I \circ H \circ {\tilde{f}}.$$
		In particular, one factor is flat.
		
		\item After possibly relabeling factors, there are isometric immersions $f_0 \colon M^{n_0}_0 \rightarrow \QH^{n_0+s_0}_{-c}$ and $f_1\colon M^{n_1}_1 \rightarrow \mathbb{S}^{n_1+s_1}_c$ such that $s_0+s_1 = 1$ and
		$$f = \Theta \circ (f_0 \times f_1),$$ where $\Theta\colon \QH^{n_0+s_0}_{-c} \times \mathbb{S}_c^{n_1+s_1} \rightarrow \R^{n_0+n_1+1}$ is a conformal representation of $\R^{n+1}$. In particular, one factor has constant and nonzero curvature.
	\end{enumerate}
\end{main}

When the number of factors increases, the existence of the immersion in the minimal codimension imposes strong restrictions to the manifold $M^n$ (See Theorem \ref{mainMinimal} in Section \ref{remarks}). Thereby, we had to add a hypothesis to obtain a more interesting theorem. Namely,

\begin{main}\label{mainKfactors}
	Let $f\colon M^n \rightarrow \R^{n+k}$ be a conformal immersion of a Riemannian product manifold $M^n =\prod_{i=0}^{k} M^{n_i}_i $, where $n =\sum_{i=0}^{k} n_i$ and $n_i \geq 2$ for all $0 \leq i \leq k$. If the subset of points of $M^{n_i}_i$ at which it is flat has empty interior for all $1 \leq i \leq k$, and at all points in $M^{n_0}_0$ there is an plane with vanishing sectional curvature, then one of the following holds
	\begin{enumerate}[(i)]
		\item There exist an extrinsic product $\tilde{f}\colon M^n \rightarrow \R^{n+k}$ of isometric immersions, a homothety $H$ of $\R^{n+k}$ and an inversion $I$ in $\R^{n+k}$ with respect to a sphere of unit radius such that
		$$f = I \circ H \circ {\tilde{f}}.$$ In particular, $M_0^{n_0}$ is lat.
		\item There are isometric immersions $f_{i_0}\colon M^{n_{i_0}}_{i_0} \rightarrow \QH^{n_{i_0}+s_{i_0}}_{-c}$ and $f_i\colon M^{n_i}_i \rightarrow \mathbb{S}^{n_i+s_i}_{c_i}$, for $i \in \{0,\cdots k\} \setminus \{i_0\}$, such that $\sum_{i=0}^{k}s_i = 1$, $\sum_{i=0, i \neq i_0}^{k} \frac{1}{c^2_i} = \frac{1}{c^2}$ and
		$$f = \Theta \circ \left(f_0 \times \tilde{f}\right),$$ where $\Theta \colon \QH^{n_{i_0}+s_{i_0}}_{-c} \times \mathbb{S}^{n+k-n_{i_0}}_{c} \rightarrow \R^{n+k}$ is a conformal representation, and $\tilde{f}\colon \prod_{i=0 , i\neq i_0}^{k} M^{n_i}_i \rightarrow \mathbb{S}^{n-n_{i_0}-s_{i_0} +k}_c$ is an extrinsic product. In particular, $M^{n_i}_i$ has constant and nonzero curvature, for all $1 \leq i \leq k$.
	\end{enumerate}
\end{main}

\section{Preliminaries}\label{prelim}

We will establish some notations and give some definitions used in the previous works in this subject. We will use a decomposition theorem due to Tojeiro, an algebraic tool called flat bilinear forms due to Moore and the light-cone representative of conformal immersions.

\subsection{Extrinsic decomposition of conformal immersions}

Lets consider an isometric (conformal) immersion of Riemannian products into the Euclidean space $f\colon M^n = \prod_{i=0}^{k} M^{n_i}_i \rightarrow \R^{n+p}$, with $n = \sum_{i=0}^{k} n_i$. We say that the second fundamental form $\alpha^f$ of $f$ is adapted to the product net of $M^n$ if $\alpha^f(X_i, X_j) = 0,$ for all $X_i \in TM_i$ such that $i\neq j.$ The notion of adapted can be extend to a decomposition of the tangent space, depending on the structure of this decomposition we will have warped products, or more general structures as the partial tubes (see \cite{tojeiroRhamconf}, \cite{partialTubes} and \cite{DaTo2018} for more details). 

We say that $f$ is an extrinsic product if $f = f_0 \times \cdots \times f_k$, where $f_i\colon M^{n_i}_i \rightarrow \R^{n_i+p_i}$ are isometric immersions and $\sum_{i=0}^{k}p_i = p$.
For an isometric immersion in the sphere, $g\colon M^n = \prod_{i=0}^{k} M^{n_i}_i \rightarrow \mathbb{S}^{n+p}_{c}$, $n = \sum_{i=0}^{k} n_i$, we also will use the term extrinsic product when $g = h \circ (g_1 \times \cdots \times g_k)$, where $g_i\colon M^{n_i}_i \rightarrow \mathbb{S}^{n_i+s_i}_{c_i}$ are isometric immersions and $h\colon \prod_{i=1}^{k}\mathbb{S}^{n_i+s_i}_{c_i} \rightarrow \mathbb{S}^{n+p}$ is the natural inclusion such that $\sum_{i=0}^{k}\frac{1}{c^2_i} = \frac{1}{c^2}$ and $\sum_{i=0}^{k}s_i = p - k + 1$.

We can now state the decomposition theorem that we will make use in this work:

\begin{thm}[Tojeiro \cite{tojeiroConfNolk}]\label{tojeiroRham}
	Let $f\colon M^n \rightarrow \R^{n+p}$, $n \geq 3$, be a conformal immersion of a Riemannian product manifold $M^n =\prod_{i=0}^{k} M^{n_i}_i $, where $n =\sum_{i=0}^{k} n_i$. If the second fundamental form of $f$ is adapted to the product net of $M^n$, then one of the following possibilities holds:
	\begin{enumerate}[(i)]
		\item There exist an extrinsic product $\tilde f\colon M^n \rightarrow \R^{n+p}$ of isometric immersions, a homothety $H$ of $\R^{n+p}$ and an inversion $I$ in $\R^{n+p}$ with respect to a unit sphere such that
		$$f = I \circ H \circ \tilde f.$$
		\item After possibly rearranging the factors, there is an isometric immersion $f_0\colon M^{n_0}_0 \rightarrow \QH^{n_0+s_0}_{-c}$ and isometric immersions  $f_i\colon M^{n_i}_i \rightarrow \mathbb{S}^{n_i+s_i}_{c_i}$, $1 \leq i \leq k$, such that $\sum_{i=0}^{k}s_i = p-k+1$, $\sum_{i=1}^{k} \frac{1}{c^2_i} = \frac{1}{c^2}$ and
		$$f = \Theta \circ \left(f_0 \times \tilde f\right),$$
		where $\Theta\colon \QH^{n_0+s_0}_{-c}\times \mathbb{S}^{n+p-n_0-s_0}_{c}$ is a conformal representation, and $\tilde{f}= f_1 \times \cdots \times f_k \colon \prod_{i=0}^{k} M^{n_i}_i \to \prod_{i=1}^k \mathbb{S}_{c_i}^{n_i+s_i} \subset \mathbb{S}_c^{n+p-n_0-s_0}$ is an extrinsic product.
	\end{enumerate}
\end{thm}

\subsection{Flat bilinear forms}

Following the ideas of the article \cite{Moore1971}, we will use some results about \textit{flat} bilienar forms, which we will define next. Let $W^{p,1}$ be a Lorentzian vector space of dimension $p+1$ endowed with a indefinite inner product $\interno{\cdot}{\cdot}$ with index $1$. Let $V$ be a finite dimensional vector space. A bilinear form $\beta\colon V \times V \rightarrow W^{p,1}$ is said to be \textit{flat} if  
\begin{equation}\label{defFlat}
	\interno{\beta(W,X)}{\beta(Y,Z)} - \interno{\beta(Y,X)}{\beta(W,Z)} = 0,
\end{equation}
for all $W,X,Y,Z \in V$. Flat bilinear forms were introduced by Moore in \cite{Moore1977} to study isometric immersions of the round sphere in Euclidean space in low codimension, and was used afterwards in several papers about isometric rigidity (see \cite{DFGenuines} and \cite{DaTo2018} for more information).

Denote the \textit{nullity} of $\beta$ by $\mathcal{N}(\beta) = \set{X \in V}{\beta(X,Y)=0 \ \text{for all} \ Y \in V}$ and the span of $\beta$ by $\mathcal{S}(\beta) = \spn\set{\beta(X,Y)}{X,Y \in V}$.

\subsection{The light-cone representative}\label{lightCone}

Let $\mathbb{L}^{n+2}$ be the $(n+2)$-dimensional Minkowski space endowed with a Lorentz scalar product of signature $(+,\cdots,+,-)$, and let $\mathbb{V}^{n+1} = \set{x \in \mathbb{L}^{n+2}}{\interno{x}{x} = 0}$ denote the light cone in $\mathbb{L}^{n+2}$. Then $\mathbb{E}^{n}= \mathbb{E}^{n}_w = \set{x \in \mathbb{V}^{n+1}}{\interno{x}{w}=1}$ is a model of $n$-dimensional Euclidean space for any $w \in \mathbb{V}^{n+1}$. Namely, choose $x_0 \in \mathbb{E}^{n}$ and a linear isometry $A \colon \R^n \rightarrow \spn\{x_0,w\}^\bot$. Then the triple $(x_0,w,A)$ gives rise to an isometry $\Psi = \Psi_{x_0,w,A}: \R^n \rightarrow  \mathbb{E}^n$ defined by $\Psi(x) = x_0 + A(x) - \frac{1}{2} \| x \|^2 w$. In view of this, any conformal immersion $f\colon M^n \rightarrow \R^{n+p}$ with conformal factor $\varphi \in C^{\infty}(M)$ gives rise to an isometric immersion $F\colon M^n \rightarrow \mathbb{V}^{n+p+1}$ given by $F = \frac{1}{\varphi} \Psi \circ f.$ The isometric immersion $F$ is called the light cone representative of $f$.

In this work we will make use the expression of the second fundamental form of $F$. The normal bundle $T_F^\bot M$ of $F$ decomposes orthogonally as $T^\bot_FM = d\Psi(T^\bot_fM) \oplus \mathbb{L}^2$, where $\mathbb{L}^2$ is the Lorentzian plane bundle spanned by the position vector $F$ and the vector $\zeta = \varphi w - d(\Psi \circ f)(\text{grad} \varphi^{-1})$. The second fundamental form of $F$ splits accordingly as
\begin{equation}\label{secondF}
	\alpha^F(X,Y) =  \varphi^{-1}d\Psi(\alpha^f(X,Y)) + \varphi \text{Hess}\varphi^{-1}(X,Y)F - \interno{X}{Y}\zeta,
\end{equation}
where Hess and grad are calculated with respect to the metric $\interno{\cdot}{\cdot}_f$ induced by $f$. Note that $\| \zeta \| = 0$ and $\interno{F}{\zeta} = 1$. It is also important to point out that if $\alpha^F$ is adapted, then $\alpha^f$ is also adapted. We refer the reader to \cite{hj} or \cite{DaTo2018} for further details.

\section{Proofs}

Lets suppose that $U_0\times U_1 \subset M_0^{n_0} \times M_1^{n_1}$ is an open subset where for every $x =(x_0,x_1) \in U_0\times U_1$ and for every pair of planes $V_0^2 \subset T_{x_0} M_0$ and $V_1^2 \subset T_{x_1} M_1$, the sum of the curvatures of $V_0$ and $V_1$ is zero, that is $K_0 + K_1 = 0$, where $K_i$ is the sectional curvature of $V_i^2$. Then, fixing $x_0 \in U_0$ and varying $x_1$ we get that the $U_1$ has constant curvature $K_1$, and analogously $U_0$ has constant curvature $K_0$. In particular $U_0 \times U_1$ is conformally flat, since $U_0$ and $U_1$ have constant and opposite curvatures. Therefore the set of points $x \in M_0\times M_1$ where there are planes $V_i^2 \subset T_{x_i} M_i$ with $K_0+K_1 \ne 0$ is open and dense. Thus Theorem A follows directly from Theorem \ref{tojeiroRham} and the following proposition.

\begin{prop}\label{prop2factors}
	Let $f \colon M^n \rightarrow \R^{n+1}$ be an conformal immersion of a Riemannian product $M^n = M^{n_0}_0 \times M^{n_1}_1$, with $n_1, n_2 \geq 2$. If there are planes $V_0^2 \subset T_{x_0} M_0$ and $V_1^2 \subset T_{x_1}M_1$ such that $K_0 + K_1 \ne 0$, where $K_i$ is the sectional curvature of $V_i$, then $\alpha^f$ is adapted to the product net of $M^n$ at $x = (x_0, x_1)$.
\end{prop}

\begin{proof}
	We first prove the result when $n_0 = n_1 = 2$. At $x \in M^n$, take an orthonormal basis $e_0,e_1,e_2, e_3$ of $T_xM$ such that $\spn \{e_{2i},e_{2i+1}\} = V^2_i = T_{x_i}M_i$ at $x_i \in M_i$ for $0 \leq i \leq 1$. Let $\omega^0,\omega^1,\omega^2, \omega^{3}$ denote the dual basis and $K_i$ the sectional curvature of $M_i^{n_i}$ at $x_i$, $0 \leq i \leq 1$, and $V^{4} = T_xM^4$. By the hypothesis we have that $K_0 + K_1 \neq 0$.
	
	Let $B_0, B_1 \colon V^{4} \times V^{4}  \to \R$ be the symmetric bilinear forms defined by
	\begin{align*}
		B_0 &= \frac{K_0}{2} \left(\omega^0 \otimes \omega^0 + \omega^1 \otimes \omega^1 - \omega^2 \otimes \omega^2 - \omega^3 \otimes \omega^3 \right); \\
		B_1 &=	\begin{cases}
			\sqrt{-(K_0 + K_1)} \left(\omega^2 \otimes \omega^2 + \omega^3 \otimes \omega^3 \right), & \text{if $K_0 + K_1<0$,}  \\
			\sqrt{K_0 + K_1} \left(\omega^2 \otimes \omega^2 - \omega^3 \otimes \omega^3 \right), & \text{if $K_0 + K_1>0$.}
		\end{cases}
	\end{align*}
	
	Let $F \colon M^n \to \mathbb{V}^{n+2} \subset \mathbb{L}^{n+3}$ be the light-cone representative of $f$ (briefly discussed in Section \ref{lightCone}), and let $\alpha^F$ be its second fundamental form, given by \eqref{secondF}. Consider $W^{3,1} = T_F^\perp M \oplus \R^1$, and the bilinear form $T\colon V^4 \times V^4 \to W^{3,1}$ given by
	$$T(X,Y) = \left(\alpha^F(X,Y) + B_0(X,Y)F \right) \oplus B_1(X,Y).$$	Since $\interno{T(X,Y)}{F} = - \interno{X}{Y}$ for all $X,Y \in V^{4}$, the nullity of $T$ is trivial, that is, $\mathcal{N}(T) = \{0\}$. It is easy to check that $T$ is flat, we only need to verify \eqref{defFlat} for elements in the basis $\{e_0,e_1,e_2,e_3 \}$.
	
	The goal is to prove that $T$ is adapted to the product net of $M^n$ and by construction implies that $\alpha^F$ is also adapted, and as observed in Section \ref{lightCone} it follows that $\alpha^f$ will be also adapted. 
	
	Suppose that $\mathcal{S}(T)$ is nondegenerated, then Corollary 2 an Theorem 2 in \cite{Moore1977} there exists a basis $\{X_0,X_1,X_2,X_3\}$ of $V^{4}$ that diagonalizes $T$. We can easily check that the nulity of $B_1$ is spanned by a subset of this basis. Since this basis is orthogonal by $\interno{T(X_i,X_j)}{F} = - \interno{X_i}{X_j}$, for $0 \leq i,j \leq 3$, and	$$V^2_0 = \mathcal{N}(B_1), \,\,V^2_1 = {V^2_0}^\bot,$$
    it follows that $V^2_1$ is also spanned by a subset of this basis. Hence, we can arrange that $X_{2i}, X_{2i+1} \in V^2_i$ for $ 0 \leq i \leq 1$, and we conclude the proposition in this case. 
	
	Let us assume now that $\mathcal{S}(T)$ is degenerated, that is, $\mathcal{S}(T) \cap \mathcal{S}(T)^\bot = \spn \{\nu\}$ for some $\nu \neq 0 \in T^\bot_FM \oplus \R$. Since $F \notin \mathcal{S}(T)^\bot$, the vectors $\{\nu,F\}$ are linearly independent and by definition $\| \nu \| = \| F \| = 0$. Therefore we can assume that $\interno{\nu}{F} = -1$. Let $T'\colon V^4 \times V^4 \rightarrow W^{3,1}$ be the flat bilinear form given by
	$$T' = T - \interno{\cdot}{\cdot} \nu,$$
	note that $\mathcal{S}(T')$ is Riemannian by construction and $\dim \mathcal{N}(T') \geq 2$ by Corollary 1 in \cite{Moore1977}. Let $\nu = \nu_F \oplus a$ be the decomposition in $ T_F^\perp M\oplus \R$. Observe that the nullity of $B_1 - \interno{\cdot}{\cdot} a$	is a subset of $V^2_1$ if $a \neq 0$, and a subset of $V^2_0$ if $a = 0$. In particular, $\mathcal{N}(T')=V^2_i$ for some $i \in \{0,1\}$ and this implies that $T$ is adapted to the product net of $M^n$ and the proposition follows.
	
	To treat the general case, choose orthonormal vectors $e_0,e_1,e_2,e_3 \in T_xM$ such that $e_{2i},e_{2i+1}$ span a plane $V^2_i \subset T_{x_i}M_i$ and $K_0 + K_1 \neq 0$, where $K_i$ is the curvature of the planes $V^2_i$ at $x_i \in M_i^{n_i}$. So, by the previous arguments we have that $\alpha^F|_{V_0 \times V_1} \coloneqq \alpha^F|_{(V_0 \times V_1) \times (V_0 \times V_1)}$ is adapted. Let $v_i \in V_i(x_i) \subset T_{x_i}M_i$, define $V^2_i({v_i})$ as the plane spanned by $\{ e_{2i} + \epsilon v_i, e_{2i+1} \}$, where $\epsilon \in \R$ is small enough such that $K_{V^2_0({v_0})} + K_1 \neq 0$, $K_0 + K_{V^2_1(v_1)} \neq 0$ and $K_{V^2_0(v_0)} + K_{V^2_1(v_1)} \neq 0$. We can now use the same argument as before for the pairs of planes $\left\{V^2_0(v_0),V_1\right\}$, $\left\{V_0,V^2_1(v_1)\right\}$ and $\{V^2_0(v_0),V^2_1(v_1)\}$ at $x \in M^n$. This implies that $\alpha^F_{V_0^{v_0} \times V_1}$, $\alpha^F_{V_0 \times V_1^{v_1}}$ and $\alpha^F_{V_0^{v_0} \times V_1^{v_1}}$ are adapted in their respective domains and, in particular, $\alpha^F(v_0,v_1) = 0$. The proposition follows by the fact that $v_i$ was arbitrary for $0 \leq i \leq 1$.
\end{proof}

To prove Theorem \ref{mainKfactors} we will need an algebraic proposition as well. The following proposition is very similar to Proposition \ref{prop2factors}, the difference is that the additional hypothesis is not an open property. Therefore, we can not make an perturbation argument directly as in the above proposition.

\begin{prop}\label{propkfactors}
	Let $ f \colon M^n \to \R^{n+k}$ be a conformal immersion of a Riemannian product $M^n = \prod_{i=0}^k M_i^{n_i}$, with $n_i \geq 2$ for all $0 \leq i \leq k$. If $M^{n_0}_0$ admits a plane with vanishing sectional curvature at $x_0$ and each $M_i^{n_i}$ is not flat at $x_i$, for $1 \leq i \leq k$, then the second fundamental form of $f$ is adapted to the product net of $M^n$ at $x=(x_0, \cdots, x_k)$.
\end{prop}

\begin{proof}
	We first prove the result when $n_i = 2$ for $0 \leq i \leq k$. At $x \in M^n$, take an orthonormal basis $e_0,\cdots, e_{2k+1}$ of $T_xM$ such that $\spn\{e_{2i},e_{2i+1}\} = V^2_i = T_{x_i}M_i$, $0 \leq i \leq k$. Let $\omega^0,\cdots, \omega^{2i+1}$ denote the dual basis and $K_i$ the sectional curvature of $M_i^{n_i}$ at $x_i$, $0 \leq i \leq k$, and $V^{2k+2} = T_xM$. Observe that $K_0=0$ and $K_i \neq 0$ for $1 \leq i \leq k$.
	
	Let $B_i \colon V^{2k+2} \times V^{2k+2}  \to \R$ be the symmetric bilinear form defined by
	$$B_i = \begin{cases}
		\sqrt{-K_i}\left( \omega^{2i}\otimes\omega^{2i} + \omega^{2i+1}\otimes\omega^{2i+1} \right), & \text{if } K_i < 0; \\
		\sqrt{K_i}\left( \omega^{2i}\otimes\omega^{2i} - \omega^{2i+1}\otimes\omega^{2i+1} \right), & \text{if } K_i > 0.
	\end{cases}$$
	
	Let $F \colon M^n \to \mathbb{V}^{n+k+1} \subset \mathbb{L}^{n+k+2}$ be the light-cone representative of $f$ (briefly discussed in Section \ref{lightCone}), and let $\alpha^F$ be the second fundamental form given by \eqref{secondF}. Consider $W^{2k+1,1} = T_F^\perp M \oplus \R^k$, $B = (B_1, \cdots, B_k) \colon V^{2k+2}\times V^{2k+2} \to \R^k$, and the bilinear form
	\begin{equation}\label{flatTform}
 	   T = \alpha^F \oplus B \colon V^{2k+2} \times V^{2k+2} \longrightarrow W^{2k+1,1}.
	\end{equation}
	It is a straightforward calculation to check that \eqref{defFlat} holds for $T$ for the basis $\left\{e_0, \cdots, e_{2k+1} \right\}$, that is, $T$ is flat. It follows from $\interno{T(X,Y)}{F} = -\interno{X}{Y}$, for all $X,Y \in V^{2k+2}$, that $\mathcal{N}(T) = \{0\}$. 

	As in Proposition \ref{prop2factors}, the goal is to prove that $T$ is adapted to the product net of $M^n$ and by construction it will follow that $\alpha^F$ and $\alpha^f$ are also adapted. 

	Suppose that $\mathcal{S}(T)$ is nondegenerated. Thus, $W^{2k+1,1} = \mathcal{S}(T)$ by Corollary 2 in \cite{Moore1977}, otherwise $\mathcal{N}(T) \neq \{0\}$. By Theorem 2 in \cite{Moore1977} there is a basis $\{X_0, \cdots, X_{2k+1}\}$ of $V^{2k+2}$ such that
	$$T(X_i,X_j) = 0, \text{ for all} \,\, i\neq j.$$
	In particular, $B(X_i,X_j) = 0, \text{ for all} \,\, i\neq j$. Therefore it is easy to check that the nullity subspace of each $B_i$, $1 \leq i \leq k$, is spanned by a subset of this basis. Since $V^2_0 = \bigcap_{j=1}^k \mathcal{N}(B_j)$, $V^2_0$ is also spanned by a subset of this basis. Observe that the basis is orthogonal by $\interno{T(X_i,X_j)}{F} = -\interno{X_i}{X_j}$.  Thus, ${V^2_{0}}^\bot$ is also spanned by a subset of this basis. Since
	$$V^2_i = \left(\bigcap_{\substack{j=1, \\ j \neq i}}^k \mathcal{N}(B_j)\right) \cap {V^2_0}^\bot \text{, for all }\,\, 1\leq i \leq k,$$
	it follows that $V^2_i$ is spanned by a subset of this basis. Changing the order of the vectors we can assume that $V^2_i = \spn \{X_{2i}, X_{2i+1}\}$ for $0\leq i \leq k$ and the proposition follows in the nondegenerated case. 

	Suppose now that $\mathcal{S}(T)$ is degenerated, that is, $\mathcal{S}(T) \bigcap \mathcal{S}(T)^\bot = \spn\{\nu\}$ for some $\nu \neq 0$ and $\| \nu \| = 0$. Since $F \notin \mathcal{S}(T)^\bot$, the vectors $\{\nu,F\}$ are linearly independent and we can assume that $\interno{\nu}{F} = -1$. Consider the flat bilinear map 
	\begin{equation}\label{tlinha}
 	   T' = T - \interno{\cdot}{\cdot} \nu,
	\end{equation}
	so $\mathcal{S}(T') \subset \spn\{\nu,F\}^\bot$ is Riemannian by construction. Using Corollary 1 in \cite{Moore1977} it follows that 
	\begin{equation}\label{isoDes}
 	   \dim \mathcal{N}(T') \geq 2k+2- \dim \mathcal{S}(T') \geq 2.
	\end{equation}
	Consider now the decomposition $\nu = \nu_F \oplus v$ in $W^{2k+1,1} = T_FM^{k+1,1} \oplus \R^{k}$, where $v = (v_1,\cdots,v_k)$, and the bilinear forms $B_i' = B_i - \interno{\cdot}{\cdot} v_i$. Thus, $\mathcal{N}(B_i') \subset V^2_i$ if $v_i \neq 0$ and $B_i' = B_i$ if $v_i = 0$. Since $\mathcal{N}(T') \subset \bigcap_{j=1}^k \mathcal{N}(B_j')$, it follows from \eqref{isoDes} that $\mathcal{N}(T') = V^2_{j_0}$ for some $j_0 \in \{0,1,\cdots,k\}.$ Observe that if $v_i = 0$ for all $1 \leq i \leq k$, then $\bigcap_{j=1}^k \mathcal{N}(B_j') = V^2_0$.  Using similar arguments to the nondegenerated case, it follows that $T'|_{{V^2_{j_0}}^\bot \times {V^2_{j_0}}^\bot}$ is adapted. In short, we have the following
	$$T'(Y, Z) = T(Y, Z), \text{ for all} \,\, Y \in V^2_{i}, Z \in {V^2_{i}}^\bot \ \text{and for all $i$},$$
	$$T(X, Z) = \interno{X}{Z}\nu, \text{ for all} \,\, X \in V^2_{j_0}, Z \in V^{2k+2},$$
	and, in particular, $T$ is adapted and the proposition follows.

	Note that in both cases, nondegenerated and degenerated, we have that
	\begin{equation}\label{diagFirstFactor}
 	   T(X_0,X_1) = \interno{X_0}{X_1}\nu,
	\end{equation} 
	for some basis $\{X_0, X_1\} \subset V_0^2$, where $\| \nu \| = 0$ ($\nu=0$ in the nondegerated case). Using Gauss equation and \eqref{diagFirstFactor}, we also have the following
	\begin{equation}\label{perp}
		\interno{\alpha^F(X_0,X_0)}{\alpha^F(X_1,X_1)} = \left\| \alpha^F(X_0,X_1)\right\|^2 = 0.
	\end{equation}

	To treat the general case, choose orthogonal vectors $e_2, \cdots,e_{2k+1} \in T_xM$ such that $e_{2j}, e_{2j+1}$ span a plane $V^2_j \subset T_{x_j} M_j$ with nonvanishing sectional curvature, for all $1 \leq j \leq k$. We can assume that
	\begin{equation}\label{weirdAssumption}
		\left\| \alpha^F(e_{2j}, e_{2j}) \right\| \neq 0 \text{ or}\,\, \interno{\alpha^F(e_{2j}, e_{2j})}{\alpha^F(e_{2j+1}, e_{2j+1})} \neq 0,    
	\end{equation}
	for all $1 \leq j \leq k$, otherwise the curvature of $V_j^2$ would be 0. Let $e_{0}, e_{1} \in T_{x_0}M_0$ be vectors spanning a plane $V^2_0 \subset T_{x_0} M_0$ with vanishing sectional curvature and satisfying \eqref{perp}. We can choose such basis since $\alpha^F|_{V\times V}$, where $V = V_0 \times \cdots \times V_k$, will be adapted by the previous arguments. More than that, we can also prove that $\alpha^F|_{\overline{V}\times \overline{V}}$ is adapted, where $\overline{V} = V_0 \times TM_1 \times \cdots \times TM_k$, using a perturbation argument, as in Proposition \ref{prop2factors}, for the planes $V^2_i$, for $1 \leq i \leq k$.

	Using Gauss equation and the fact that $\alpha^F|_{\overline{V}\times \overline{V}}$ is adapted, it follows that
	$$\interno{\alpha^F(e_{2l,2l})}{\alpha^F(e_{2j},e_{2j})} = 0 \text{ for $j\neq l$, $0 \leq j,l \leq k$},$$
	and from \eqref{weirdAssumption} it follows that the vectors $\alpha^F(e_{2j},e_{2j})$ are linear independent.
	
	Let $W_0 = \spn\set{\alpha^F(e_{2j},e_{2j})}{0 \leq j \leq k}$, thus $W_0^\bot = \spn\left\{\alpha^F(e_1,e_1)\right\}$, by \eqref{perp}. For $ 1 \leq j \leq k$, it is a straightforward calculation to show that
	$$\left\{\alpha^F(e_{2i}, e_{2i})\right\}_{\substack{0 \leq i \leq k,\\ i \ne j}}^\perp = \spn \left\{\alpha^F(e_1,e_1), \alpha^F(e_{2j}, e_{2j}) \right\}$$
	and it follows that $\alpha^F(e_{2j+1},e_{2j+1}), \alpha^F(e_{2j},e_{2j+1}) \in \spn\left\{\alpha^F(e_1,e_1), \alpha^F(e_{2j},e_{2j})\right\}$, $1 \leq j \leq k$. Thus, rotating both vectors $e_{2j},e_{2j+1}$, we can assume that $\alpha^F(e_{2j},e_{2j+1})$ and $\alpha^F(e_{1},e_{1})$ are parallel. Thus, it follows that  $\alpha^F(e_{2j},e_{2j+1}) = 0$, $1 \leq j \leq k$, otherwise $\interno{\alpha^F(e_{2j},e_{2j+1})}{F} \neq 0$.
	
	Analogously, we also have the $(k+1)$-dimensional subspace
	$${W_1} = \spn\set{ \alpha^F(e_{2j+1},e_{2j+1})}{0 \leq j \leq k}$$
	such that $W_1^\bot = \spn\left\{\alpha^F(e_0,e_0)\right\}$.

	We already know that $\alpha^F(X,Y) = 0$ if $X \in T_{x_i} M_i$ and $Y \in TM_j(x)$ for $1 \leq i,j \leq k$, and $i \neq j$. Let $X \in T_{x_0} M_0$ and $l \neq 0$, then Gauss equation gives
	$$\interno{\alpha^F(X,e_{2l})}{\alpha^F(e_{2s+1},e_{2s+1})} = 0,  \text{ for all} \ 0\leq s \leq k.$$
	Thus $\alpha^F(X,e_{2l}) \in W_1^\bot$ and, since $\interno{\alpha^F(X,e_{2l})}{F} = 0$, it follows that $\alpha^F(X,e_{2l}) = 0$. Let $Y \in TM_j(x)$, $1 \leq j \leq k$, then the Gauss equation gives
	$$\interno{\alpha^F(X,Y)}{\alpha^F(e_{2l},e_{2l})} = 0, \text{ for all} \ 0 \leq l \leq k,$$
	thus $\alpha^F(X,Y) \in W_0^\bot$. Since $\interno{\alpha^F(X,Y)}{F} = 0$, it follows that $\alpha^F(X,Y) = 0$. In particular, $\alpha^F$ and $\alpha^f$ are adapted to the product net of $M^n$ at $x$, and the proposition follows.	
\end{proof}

\section{Concluding remarks}\label{remarks}

One way to get a codimension restriction for a conformal immersion it is using the flat bilinear form $T$ given by \eqref{flatTform}. In deed, let $f\colon M^n =\prod_{i=0}^{k} M^{n_i}_i \rightarrow \R^{n+p}$ be a conformal immersion, where $n = \sum_{i=0}^{k} {n_i}$ and $n_i \geq 2$ for $0 \leq i \leq k$. Suppose that $f$ satisfies the hypothesis of Proposition \ref{propkfactors}, but without any restriction over $M^{n_0}_0$. In the proof of Proposition \ref{propkfactors}, instead of taking a plane in the factor $M^{n_0}_0$, choose an arbitrary direction $e_0 \in T_{x_0} M_0$ and denote the span of $\{e_0\}$ by $V^1_0$. We can define a flat bilinear form $T\colon V^1_0\times V^2_1 \times \cdots \times V^2_k \rightarrow (T_F^\bot M)^{p+1,1}\oplus \R^{k}$ identical to \eqref{flatTform} and we will end in one of the following cases:
\begin{enumerate}[(i)]
    \item When $\mathcal{S}(T)$ is nondegenerated, it follows from Corollary 2 in \cite{Moore1977} that $\dim \mathcal{N}(T) = 0 \geq 2k+1-(p+2+k) = k - 1-p$. In other words, $p \geq k - 1$;
    \item When $\mathcal{S}(T)$ is degenerated, we can define the flat bilinear form $T'\colon V^1_0\times V^2_1 \times \cdots \times V^2_k \rightarrow W^{p,0} \oplus \R^{k}\subset (T_F M)^{p+1,1} \oplus \R^{k}$ as in \eqref{tlinha}. It will follow that $\mathcal{N}(T') \subset V_j$ for some $0 \leq j \leq k$. Consequently, $p \geq k -1$ by Corollary 2 in \cite{Moore1977}.
\end{enumerate}
In other words $p \geq k-1$, and this estimate is sharp as can be seen in the trivial example:

\begin{example}
	For each $1 \leq i \leq k$, let $f_i\colon \mathbb{S}^{n_i}_{c_i} \hookrightarrow \R^{n_i+1}$ be the canonical isometric immersion and let $\Theta\colon \QH^{n_0}_{-c} \times \mathbb{S}^{n-n_0+k-1}_{c} \to \R^{n+k-1}$ be a conformal representation of $\R^{n+k-1}$, where $n = \sum_{i=0}^k n_i$ and $\sum_{i=1}^{k} \frac{1}{c_i^{2}} = \frac{1}{c^2}$. Then $f = \Theta \circ \left(I_d \times \tilde f \right) \colon \QH^{n_0}_c \times \prod_{i=1}^{k} \mathbb{S}^{n_i}_i \rightarrow \R^{n+k-1}$ is a conformal immersion in codimension $k-1$, where $I_d$ is the identity of $\QH^{n_0}_{-c}$ and $\tilde f = f_1 \times \cdots \times f_k \colon \prod_{i=1}^{k} \mathbb{S}^{n_i}_{c_i} \rightarrow \mathbb{S}_c^{n-n_0+k-1} \subset \prod_{i=1}^{k} \R^{n_i+1}$ is the natural extrinsic product.
\end{example}

Moreover, using the flat bilinear form described in this section in the proof of Proposition \ref{propkfactors} it will follow that the immersion is adapted in the minimal codimension. Thus, by Theorem \ref{tojeiroRham} the trivial example is essentially the only one in this codimension, since the codimension is not big enough to the immersion be as in the case (i) in Theorem \ref{tojeiroRham}. More precisely, we have the following:

\begin{main}\label{mainMinimal}
	Let $f: M^n \rightarrow \R^{n+k-1}$ be a conformal immersion of a Riemannian product manifold $M^n =\prod_{i=0}^{k} M^{n_i}_i $, where $n =\sum_{i=1}^{k} n_i$ and $n_i \geq 2$ for all $0 \leq i \leq k$. If the subset of points of $M^{n_i}_i$ at which $M_i$ is flat has empty interior, for $1 \leq i \leq k$, then, after possibly relabeling factors, there are isometric immersions $f_0 \colon M^{n_0}_0 \to \QH_{-c}^{n_0}$ and $f_i \colon M^{n_i}_i \to \mathbb{S}_{c_i}^{n_i}$, $1 \leq i \leq k$, such that $\sum_{i=1}^{k}\frac{1}{c^2_i} = \frac{1}{c^2}$ and $f = \Theta \circ \left(f_0 \times \tilde f \right)$, where $\Theta \colon \QH_{-c}^{n_0} \times \mathbb{S}_{c}^{n-n_0+k-1} \to \R^{n+k-1}$ is a conformal representation and $\tilde f = f_1 \times \cdots \times f_k \colon \prod_{i=1}^k M_i^{n_i} \to \prod_{i=1}^k \mathbb{S}_{c_i}^{n_i} \subset \mathbb{S}_c^{n-n_0+k-1}$ is an extrinsic product.
\end{main}

To conclude this work, we want to make a few remarks. We believe that the additional hypothesis of Theorem \ref{mainKfactors} is not necessary. And without this assumption, the authors expect an similar classification to Theorem 14 in \cite{dajczerTojeiroWarpedCod2}. In other words, that one of the following holds:
\begin{enumerate}[(i)]
	\item $f$ will be adapted and the same conclusion of Theorem \ref{mainKfactors} holds;
	\item $f$ will be a conformal composition, that is, $f = h \circ g$ where $g:\prod_{i=0}^{k} M^{n_i}_i \rightarrow \R^{n+k-1}$ an extrinsic product like in Theorem \ref{mainMinimal} and $h \colon \R^{n+k-1} \to \R^{n+k}$ is a conformal immersion. 
\end{enumerate}

Since the techniques used in this work are algebraic, we make the assumption over the curvature in Theorem \ref{mainKfactors} to exclude the second situation above. In order to improve our results to a classification will be necessary some information about the normal connection.

One of the main inspirations of this work was the following conjecture of Moore that do not receive proper attention in the opinion of the authors.

\begin{conj}[Moore \cite{Moore1971}]\label{conjMoore}
	If $M^{n_1}_1$ (resp. $M^{n_2}_2$) is a Riemannian manifold which can be locally isometrically immersed in $\R^{n_1+p_1}$ (resp. $\R^{n_2+p_2})$ but in no lower-dimensional Euclidean space, then every isometric immersion $f\colon M^{n_1}_1\times M^{n_2}_2 \rightarrow \R^{n_1+n_2+p_1+p_2}$ is an extrinsic product.
\end{conj}

This work also suggests that a similar conjecture remains true in the conformal realm.

\bibliographystyle{abbrv}
\bibliography{bibliography}

\vspace{1eM}
{\footnotesize
\begin{tabular}{lll}
	Felippe Guimar\~aes && Bruno Mendon\c{c}a \\
	Universidade Federal de Sergipe && Universidade Estadual de Londrina \\
	Av. Marechal Rondon s/n && Rod Celso Garcia Cid PR 445 Km 380\\
	49100-000, Aracaju -- SE, Brazil && 86057-970, Londrina -- PR, Brazil \\
	\textit{e-mail:} \texttt{felippe@impa.br} && \textit{e-mail:} \texttt{brunomrs@uel.br}
\end{tabular}}

\end{document}